\documentclass[preprint,12pt]{elsarticle}

\usepackage{amssymb}
\usepackage{amsthm}
\usepackage{hyperref}
\usepackage{paralist, mathdots} 
\usepackage{float}
\usepackage{paralist} 
\usepackage{amsmath,tikz,wasysym,multirow,enumitem,lscape,multirow,subfigure}

\newtheorem{corollary}{Corollary}[section]
\newtheorem{lemma}{Lemma}[section]

\newtheorem{proposition}{Proposition}[section]
\newtheorem{theorem}{Theorem}[section]
\theoremstyle{definition}
\newtheorem{example}{Example}[section]

\newtheorem{definition}{Definition}[section]

\newcommand{\npmatrix}[1]{\left( \begin{matrix} #1 \end{matrix} \right)}
\newcommand{\R}{\mathbb{R}}

\newcommand{\mc}{\mathcal}

\newcommand{\rk}{\mathrm{rk}}

\newcommand{\rank}{\mathrm{rank}}

\newcommand{\spr}{\mathrm{st}_+}

\newcommand{\sprp}{\mathrm{st}_+^{>}}

\newcommand{\cp}{\mathrm{cp}}
\newcommand{\In}{\mathrm{In}}

\DeclareMathOperator{\supp}{supp}
\DeclareMathOperator{\diag}{diag}

\DeclareMathOperator{\vc}{vec}

\newcommand{\norm}[1]{\ensuremath{\left\|{#1}\right\|}}

\definecolor{todo}{rgb}{.2,.2,.8}

\definecolor{damjana}{rgb}{.8,.2,.2}

\definecolor{helena}{rgb}{.2,.2,.8}

\begin{document}

\begin{frontmatter}



\title{\bf Symmetric Nonnegative Matrix Trifactorization}


\author[inst1,inst2]{Damjana Kokol Bukov\v{s}ek}

\affiliation[inst1]{organization={School of Economics and Business, 
University of Ljubljana},
           country={Slovenia}}
           
\affiliation[inst2]{organization={Institute of Mathematics Physics and Mechanics},
            city={Ljubljana},
           country={Slovenia}}

\author[inst3]{Helena \v{S}migoc}

\affiliation[inst3]{organization={School of Mathematics and Statistics, University College Dublin},
           country={Ireland}}

\begin{abstract}
The Symmetric Nonnegative Matrix Trifactorization (SN-Trifactorization) is a factorization of an $n \times n$ nonnegative symmetric matrix $A$ of the form $BCB^T$, where $C$ is a $k \times k$ symmetric matrix, and both $B$ and $C$ are required to be nonnegative. This work introduces the SNT-rank of $A$, as the minimal $k$, for which such factorization exists.
After listing basic properties and exploring SNT-rank of low rank matrices, the class of nonnegative symmetric matrices with SNT-rank equal to rank is studied. The paper concludes with a completion problem, that asks for matrices with the smallest possible SNT-rank among all nonnegative symmetric matrices with given diagonal blocks. 
\end{abstract}


\end{frontmatter}

\section{Introduction and Notation}

Factorizations of matrices, where the factors are required to be entry-wise nonnegative, have seen a lot of attention in the recent years, since they provide a powerful tool in analysing nonnegative data. In parallel with applications, theoretical study of those factorizations is a vibrant topic of research, that supports development of applications. 
In this work, we consider a factorization of nonnegative  symmetric matrices, which takes into account  symmetry, nonnegativity and low rank of a matrix. 

Throughout the paper we depend on predominantly standard notation listed below.  By $\R_+$ we denote the set of nonnegative real numbers, by $\R^{n \times m}$ the set of $n \times m$ real matrices, and by $\R_+^{n \times m}$ the set of  $n \times m$ entry-wise nonnegative matrices. Our investigation is focused on symmetric nonnegative matrices. For simplicity, we will generally assume that our matrices are irreducible. To this end we define
\begin{align*}
\mc S_n^+ &= \{A \in \R_+^{n \times n}; A = A^T \},\\
\hat{\mc S}_n^+ &= \{A \in \mc S_n^+; A \text{ is irreducible}\}. 
\end{align*}

Given $\alpha_i \in \R$, $i=1,\ldots, n$,  we denote by  $\diag\npmatrix{\alpha_1 & \alpha_2 & \ldots & \alpha_n}$ the diagonal matrix with diagonal entries $\alpha_1, \alpha_2, ..., \alpha_n$. Let $A, B \in \R^{m \times n}$. Then $A \circ B$ denotes the Hadamard product of $A$ and $B$ defined by $(A \circ B)_{ij} = A_{ij}B_{ij}$. For subsets ${\cal R} \subseteq \{1, 2, \ldots, m\}, {\cal C} \subseteq \{1, 2, \ldots, n\}$, $A[{\cal R}, {\cal C}]$ denotes the submatrix of $A$ containing entries $A_{ij}$ with $i \in {\cal R}, j \in {\cal C}$ and $A[{\cal R},\_]$ the submatrix $A[{\cal R}, \{1, 2, \ldots, n\}]$. Let $c_i(A)$ denote the $i$-th column of $A$, then $\vc(A) = \left(\begin{array}{cccc}
 c_1(A)^T & c_2(A)^T & \dots & c_n(A)^T 
\end{array}\right)^T \in \R^{mn}$ is the vector with columns of $A$ arranged in a long column.

Two well known factorizations that feature nonnegative factors are Nonnegative Matrix Factorization and Completely Positive Factorization. We briefly introduce both factorizations below. 

\subsection{Nonnegative Matrix Factorization} 

Given a nonnegative $n \times m$ matrix $A$ and a positive integer $k$, \emph{the Nonnegative Matrix Factorization (NM-Factorization)} consists of finding matrices $U \in \R_+^{n\times k}$ and $V \in \R_+^{m \times k}$ so that $UV^T$ approximates $A$. The most frequently used measure of approximation is the Frobenius norm, hence the goal is to find $U$  and $V$ that minimise $\norm{A-UV^T}_F$. 

The exact version of NM-Factorization is looking for a minimal $k_0$ for which there exist matrices $U \in \R_+^{n\times k_0}$ and $V \in \R_+^{m \times k_0}$ with $A=UV^T$. We will denote such $k_0$ by $\rk_+(A)$, while $\rk(A)$ will denote the rank of $A$. Clearly, $$\rk(A) \leq \rk_+(A) \leq \min\{m,n\}.$$ 

The approximate NM-factorisation by Paatero and  Tapper \cite{https://doi.org/10.1002/env.3170050203}, has seen a tremendous growth after a seminal paper of Lee and Seung \cite{LSNature}. We refer the reader to the following recent publications \cite{8653529, NMFGillis} for background on the problem, and offer a sample of the works that consider the exact version of the problem \cite{MR1230356,leplat2021exact, MR3720357}.  

\subsection{Symmetric NM-Factorization and Completely Positive Factorization}

When dealing with nonnegative symmetric matrices, it makes sense to  look for factorizations that exhibit not only nonnegativity but also symmetry. The most influential factorization that fits this requirement is defined for completely positive matrices.

\emph{The Symmetric NM-Factorization (SN-Factorization)} is a variation of NM-Fac\-to\-ri\-zation  where $U=V$. Hence, we are looking for approximations of a given symmetric nonnegative matrix by a matrix of the form $UU^T$. 
If a matrix can be written as $UU^T$ exactly, for some nonnegative matrix $U$, then it is said to be \emph{completely positive}. We call such factorization \emph{completely positive factorization} and use the abbreviation \emph{CP-Factorization}. While completely positive matrices are necessarily positive semidefinite, not every nonnegative positive semidefinite matrix is completely positive \cite[Example 2.4]{MR1986666}.

For a completely positive matrix $A$, we define $\cp(A)$ to be the minimal $k$ such that there exist $U \in \R_+^{n \times k}$ with $A=UU^T$. If a matrix $A$ is not completely positive, we define $\cp(A)$ to be equal to infinity. 

For the background on the completely positive factorization we refer the reader to the following works \cite{MR3414584,MR1986666,MR3273762}. 

\subsection{Symmetric Nonnegative Matrix Trifactorization}

In this paper we study a factorization that exhibits nonnegativity, symmetry and low rank of a matrix.

 \emph{Symmetric Nonnegative Matrix Trifactorization} is a an approximate factorization of a given symmetric matrix $A$ of the form $BCB^T$, where $B$ and $C$ are nonnegative, and $C$ is symmetric. As above, the Frobenius norm is typically used to measure the approximation. In this work we consider the exact version of SNMT, which we refer to by the acronym \emph{SN-Trifactorization}.
 
 In contrast to extensive literature on 
NM-Factorization and  CP-Facto\-ri\-za\-tion, SN-Tri\-facto\-ri\-za\.tion has so far received considerably less attention. We refer the reader to the works  \cite{pmlr-v28-arora13, 8338424, NMFGillis} on Symmetric Nonnegative Matrix Trifactorization. This factorization is also known as Semi (or weighted) Symmetric Nonnegative Factorization
 \cite{ding-he-simon, thesis-ho, MR2252778}. Applications of SN-Trifactorization established to date include Hidden Markov Model Indentification and Community Detection \cite{8653529}.

\begin{definition}
A factorization of $A \in  \mc S_n^+$ of the form $BCB^T$, where $B \in \R_+^{n \times k}$ and $C \in \mc S_k^+$, is  called \emph{SN-Trifactorization} of $A$. Minimal possible $k$ in such factorization, is 
called the \emph{SNT-rank} of $A$, and is denoted by $\spr(A)$.
\end{definition}

 \subsection{Overview}
The paper defines (exact) SN-Trifactorization and SNT-rank and dedicates Section \ref{sec:basic} to basic properties of this newly defined parameter. Those include comparison with related parameters, and investigation which properties of the classical rank transfer to SNT-rank. Matrices of rank $2$ and of rank $3$ are examined. The class of nonnegative symmetric matrices whose rank equal SNT-rank is studied in Section \ref{sec:perturbation}. 
Section \ref{sec:constructions} is dedicated to a completion problem. The work is concluded with a handful of questions for further research.

\section{Basic Observations}\label{sec:basic}
The results in this section establish basic properties of SNT-rank. In the introduction we met three different ranks of a matrix that are defined through factorizations involving nonnegative factors: $\rk_+(A)$, $\spr(A)$ and $\cp(A)$. First, we take a look at how they compare. 

\begin{proposition}\label{prop:basic1}
Let $A$ be a nonnegative symmetric matrix. Then:
\begin{enumerate}
\item $\rk(A) \leq \spr(A) \leq n$,
\item $\rk_+(A) \leq \spr(A) \leq \cp(A)$.
\end{enumerate}
\end{proposition}

\begin{proof} Both items are quickly deduced from arguments below:
\begin{enumerate}
\item From $A=BCB^T$, $B \in \R^{n \times k}$, we get $\rk(A) \leq \rk(B) \leq k$. On the other hand, every $A \in \mathcal{S}_n^+$ can be written as $A=IAI^T$, so $\spr(A) \leq n$. 

\item Factorization $A=BCB^T$ can be considered as NM-Factorization with $U = B$ and $V^T = CB^T$. For $A$ with $\cp(A) < \infty$, we have $A = UIU^T$ and $\spr(A) \leq \cp(A)$.
\end{enumerate}
\end{proof}

Note that all the inequalities listed in Proposition \ref{prop:basic1} can be strict.

\begin{example}\label{ex:1}
The matrix $$A=\left(
\begin{array}{cccc}
 1 & 1 & 0 & 0 \\
 1 & 0 & 1 & 0 \\
 0 & 1 & 0 & 1 \\
 0 & 0 & 1 & 1 \\
\end{array}
\right),$$
also considered in \cite{MR1230356}, has $\rk(A)=3$, $\rk_+(A)=\spr(A)=4$, and $\cp(A)=\infty$. 

On the other hand, the matrix
$$\ \ B=\left(
\begin{array}{cccc}
 4 & 2 & 2 & 0 \\
 2 & 3 & 1 & 2 \\
 2 & 1 & 3 & 2 \\
 0 & 2 & 2 & 4 \\
\end{array}
\right)=\left(
\begin{array}{ccc}
 2 & 0 & 0 \\
 1 & \sqrt{2} & 0 \\
 1 & 0 & \sqrt{2} \\
 0 & \sqrt{2} & \sqrt{2} \\
\end{array}
\right)\left(
\begin{array}{cccc}
 2 & 1 & 1 & 0 \\
 0 & \sqrt{2} & 0 & \sqrt{2} \\
 0 & 0 & \sqrt{2} & \sqrt{2} \\
 \end{array}
\right)$$
has $\rk(B)=\rk_+(B)=\spr(B)=\cp(B)=3$.
\end{example}

Two simple lemmas below are referred to later in selected proofs. First one lists trivial ambiguities in the SN-Trifactorization. 

\begin{lemma}\label{lem:permutation}
Let $A=BCB^T$ be a SN-Trifactorization for $A$, $P$ a permutation matrix, and $D$ a diagonal matrix with positive diagonal elements. 
\begin{enumerate}
    \item Taking $B_1=BPD$ and $C_1=D^{-1}P^TCPD^{-1}$, we get another SN-Trifacto\-ri\-zation for $A$: $A=B_1C_1B_1^T$. 
      \item The matrix $A_2=DPAP^TD^T$ has a SN-Trifactorization $A_2=B_2CB_2^T$ for $B_2=DPB$. 
\end{enumerate}
Both factorizations exhibit the same the size of the SN-Trifactorization as the original factorization. 
\end{lemma}

\begin{lemma}\label{lem:zeroprod}
Let $B' \in \R_+^{n \times s}, B'' \in \R_+^{m \times t}, C \in \R_+^{s \times t}$ satisfy $B'C(B'')^T=0$. If neither $B'$ nor $B''$ has a zero column, then $C=0$.
\end{lemma}

\begin{proof}
From
$$(B'C(B'')^T)_{ij}=\sum_{k=1}^s \sum_{l=1}^t (B')_{ik}(C)_{kl}(B'')_{jl} =0,$$
we can conclude that
$(B')_{ik}(C)_{kl}(B'')_{jl}=0$ for any collection of four indices $i\in \{1,\ldots, n\}$, $j\in \{1,\ldots, m\}$, $k \in \{1,\ldots, s\}$, and $l\in \{1,\ldots, t\}$.
From the assumption that neither $B'$ nor $B''$ have a zero column, we deduce that for any $k \in \{1,\ldots, s\}$ there exists $i(k) \in \{1,\ldots, n\}$ with $(B')_{i(k)k}\ne 0$, and for any $l \in \{1,\ldots, t\}$ there exists $j(l) \in \{1,\ldots, m\}$ with $(B'')_{j(k)l}\ne 0$. The claim follows.
\end{proof}

Some properties of the classical rank extend to SNT-rank. 

\begin{proposition}\label{prop:basic}
Let $A, A' \in \mc S_n^+$, $A'' \in \mc S_m^+$ and let $A_0$ be a principal submatrix of $A$. Then:
\begin{enumerate}
\item  $\spr(A+A') \leq \spr(A)+\spr(A')$.
\item  $\spr(A_0) \leq \spr(A)$. 
\item $\spr(A^m) \leq \spr(A)$ for any positive integer $m$.
\item  $\spr(A'\oplus A'') = \spr(A')+\spr(A'')$.
\end{enumerate}
\end{proposition}

\begin{proof}
Inequalities can be deduced from corresponding SN-Trifactorizations as follows:
\begin{enumerate}
\item From $A=BCB^T$ and $A'=B'C'B'^T$ we get
$$ A+A' = \npmatrix{B & B'}\npmatrix{C & 0 \\ 0 & C'}\npmatrix{B^T \\ B'^T} .$$  

\item  Let $A\in \R_+^{n \times n}$ have a principal submatrix $A_0\in \R_+^{m \times m}$, $m <n$. Using Lemma \ref{lem:permutation}, we can without loss of generality assume $A = \npmatrix{A_0 & * \\ * & *}$.
Let $A=BCB^T$ with $B=\npmatrix{B_1 \\ *}\in \R_+^{n \times k}$ and $B_1\in \R_+^{m \times k}$.
From $A_0=B_1CB_1^T$, we get $\spr(A') \leq \spr(A)$.
\item  Suppose that $A=BCB^T$. Then $A^m=B(CB^TB)^{m-1}CB^T$ is a SN-Trifacto\-ri\-zation of $A^m$, since the matrix $(CB^TB)^{m-1}C$ is symmetric. \item Similarly, $A'=B'C'B'^T$ and $A''=B''C''B''^T$ gives us:
$$ \npmatrix{A' & 0 \\ 0 & A''} = \npmatrix{B' & 0 \\ 0 & B''}\npmatrix{C' & 0 \\ 0 & C''}\npmatrix{B'^T  & 0 \\ 0 & B''^T} ,$$
proving $\spr(A'\oplus A'') \le \spr(A')+\spr(A'')$. 
\end{enumerate}
As we have equality in the last item, we still need to prove converse inequality. Let $\spr(A'\oplus A'')=s$ with  a corresponding SN-Trifactorization $A'\oplus A''=BCB^T$.  Using Lemma \ref{lem:permutation}, we can assume that $B$ and $C$ are of the following form:
$$B=\npmatrix{B_{11} & 0 \\ B_{21} & B_{22}} \text{ and } C=\npmatrix{C_{11} & C_{12} \\ C_{12}^T & C_{22}},$$
where $B_{11} \in \R_+^{n \times s_1}$, $B_{22}\in \R_+^{m \times (s-s_1)}$, $C_{12}\in \R_+^{s_1 \times (s-s_1)},$ and $B_{11}$ has no zero columns. Since the factorization corresponds to $\spr(A' \oplus A'')$, $B_{22}$ also has no zero columns. From
$B_{11}C_{11}B_{21}^T+B_{11}C_{12}B_{22}^T=0$, we get $B_{11}C_{12}B_{22}^T=0$, and thus $C_{12}=0$ by Lemma \ref{lem:zeroprod}. Now, $C_{11}$ has no zero rows or columns, so $B_{11}C_{11}$ has no zero columns. Since $B_{11}C_{11}B_{21}^T=0$, we have $B_{21}=0$ again by Lemma \ref{lem:zeroprod}. We now have $A'=B_{11}C_{11}B_{11}^T$ and $A''=B_{22}C_{22}B_{22}^T$, proving $\spr(A')\le s_1$ and $\spr(A'')\le s-s_1$.
\end{proof}

\begin{proposition}\label{prop:fromNMF}
Let $A \in \mc S_n^+$. Then:
\begin{enumerate}
    \item $\spr(A) \leq 2 \rk_+(M)$ for all $M \in \R_+^{n \times n}$ satisfying $A=M+M^T$, 
    \item $\spr(A) \leq 2 \rk_+(A)$, 
    \item If 
    \begin{equation}\label{eq:AX}
         A=\npmatrix{0 & X \\ X^T & 0} \in \mc S_{n_1+m_1}^+,
     \end{equation}
    for  $X\in \R_+^{n_1 \times m_1}$, $n_1+m_1=n$, then  $\spr(A) =2 \rk_+(X)$.
\end{enumerate}
\end{proposition}

\begin{proof}
Let  $\rk_+(M)=k$, $U \in \R_+^{n \times k}$, $V \in \R_+^{n \times k}$, and  $M=UV^T$. We have 
$$A=\npmatrix{U & V} \npmatrix{0 & I_k \\ I_k & 0}\npmatrix{U^T \\ V^T},$$
proving $\spr(A) \leq 2k$. This proves the first two items in the proposition. The inequality $\spr(A) \leq 2 \rk_+(X)$ in item 3. is now also established by noting that 
$$\rk_+(X)=\rk_+\left(\npmatrix{0 & X \\ 0 & 0}\right).$$  

Let $A$ be as in \eqref{eq:AX} and $\spr(A)=s$ with corresponding SN-Trifactorization $A=BCB^T$.  Using Lemma \ref{lem:permutation}, we can assume:
$$B=\npmatrix{B_{11} & 0 \\ B_{21} & B_{22}} \text{ and } C=\npmatrix{C_{11} & C_{12} \\ C_{12}^T & C_{22}},$$
where $B_{11} \in \R_+^{n_1 \times s_1}$, $B_{22}\in \R_+^{m_1 \times (s-s_1)}$, $C_{12}\in \R_+^{s_1 \times (s-s_1)},$ and $B_{11}$ has no zero columns. From $B_{11}C_{11}B_{11}^T=0$ we get $C_{11}=0$ by Lemma \ref{lem:zeroprod}. Now $X=B_{11}C_{12}B_{22}^T$, hence $\rk_+(X) \leq \min\{s_1,s-s_1\} \leq \lfloor \frac{s}{2}\rfloor$, proving $2\rk_+(X) \leq s$, as required.
\end{proof}

As we illustrate in the example bellow, it can happen that $\spr(A)<2 \rk_+(M)$ for all nonnegative $M$ satisfying $A=M+M^T$. 

\begin{example}
Let $e = \npmatrix{1 & 1 & ... & 1}^T \in \R^n$ be the vector of all $1$'s,
$$M:=\left(
\begin{array}{ccc}
 0 & 1 & 0_{1 \times n} \\
 1 & 0 & 0_{1 \times n} \\
 e & e & 0_{n \times n}
\end{array}
\right) \in \R_+^{(n+2) \times (n+2)} \text{ and } A:=M+M^T=\left(
\begin{array}{ccc}
 0 & 2 & e^T \\
 2 & 0 & e^T \\
 e & e & 0
\end{array}
\right).$$
From $\rk(M)=\rk_+(M)=2$ and $\rk(A)=3$, we get $3 \le \spr(A) \le 4$ by Proposition \ref{prop:fromNMF}. Actually, $\spr(A) =3$, since
$A=BCB^T$ with 
$$B=\left(
\begin{array}{lll}
 1 & 0 & 0 \\
 0 & 1 & 0 \\
 0_{n \times 1} & 0_{n \times 1} & e
\end{array}
\right) \in \R_+^{(n+2) \times 3} \text{ and } C=\left(
\begin{array}{ccc}
 0 & 2 & 1 \\
 2 & 0 & 1 \\
 1 & 1 & 0
\end{array}
\right).$$
\end{example}

Separable NMF is a variation of NMF, where the columns of the first factor in $A=UV^T$ are chosen from the columns of the matrix $A$, \cite{NIPS2003_1843e35d}. It turns out that with separability condition added, SNT-rank and NMF-rank agree.  

\begin{lemma}\label{lem:sep}
Let $A \in\mc S_n^+$ be a nonnegative symmetric matrix and $P \in \R^{n \times n}$ a permutation matrix. If 
\begin{equation}\label{eq:partition}
P^TAP=\npmatrix{A_{11} \\ A_{21}}\npmatrix{I_k & Q}
\end{equation}
for some nonnegative matrix $Q \in \R_+^{k \times (n-k)}$, then $\spr(A)\leq k$. 
\end{lemma}

\begin{proof}
Since $A=A^T$ we have
$P^TAP=\npmatrix{I_k \\ Q^T} A_{11} \npmatrix{I_k & Q}$. 
\end{proof}

If the condition that $Q$ is nonnegative is removed, then a factorization of the form \eqref{eq:partition} exists for every symmetric matrix $A$, with $k=\rk(A)$. In particular, any rank $1$ matrix $A \in \mc S_n^+$ can be written as $A = vv^T$ for some $v \in \R^n_+$, thus $\spr(A) = 1$. The following corollary proves that a similar conclusion is true also for matrices of rank $2$. 
A similar results holds for NMF-rank, see \cite[Theorem 4.1]{MR1230356}.  
 
\begin{corollary}\label{thm:rk2}
Let $A \in \mc S_n^+$ be a nonnegative symmetric matrix of rank $2$. Then $\spr(A)=2$. 
\end{corollary}

\begin{proof}
Every rank $2$ matrix is separable by the proof of Theorem 2.6 in \cite{NMFGillis}. Using Lemma \ref{lem:sep} we get $\spr(A)=2$.
\end{proof}

Corollary \ref{thm:rk2} cannot be generalised to matrices with $\rk(A)=3$, or even to matrices with $\rk_+(A)=3$. This is shown in our next example, that also illustrates that $\spr(A)>\rk_+(A)$ can happen, showing that $\spr(A)$ is indeed a new parameter. 

\begin{example}\label{ex:2}
Let 
\begin{equation}\label{eq:ABC}
A=B_1C_1=\left(
\begin{array}{ccc}
 0 & 1 & 1  \\
 0 & 0 & 1  \\
 1 & 0 & 0 \\
 1 & 1 & 0 \\
\end{array}
\right)\left(\begin{array}{cccc}
 2 & 1 & 0 & 1  \\
 0 & 1 & 1 & 0  \\
 1 & 0 & 1 & 2  \\
\end{array}
\right)=\left(\begin{array}{cccc}
 1 & 1 & 2 & 2  \\
 1 & 0 & 1 & 2  \\
 2 & 1 & 0 & 1  \\
 2 & 2 & 1 & 1 \\
\end{array}
\right).
\end{equation}
Clearly,  $\rk(A)=\rk_+(A)=3$ and $\spr(A) \in \{3,4\}$. Assuming $\spr(A) = 3$, we have
$A=BCB^T$ with $B\in \R_+^{4 \times 3}$, $C \in \R_+^{3 \times 3}$, and $\rk(C)=3$. 

With the aim of arriving at contradiction, we first we consider the pattern restrictions on $B$ and $C$ coming from the two zero entries in $A$. Let us denote the rows of $B$ by $b_i^T$, $i=1,2,3,4$.  Hence, $b_2^TCb_2=b_3^TCb_3=0$. If either $b_2$ or $b_3$ have two positive entries, then $C$ needs to have a $2 \times 2$ principal submatrix  equal to zero, contradicting $\rk(C)=3$. Now that we know that $b_2$ and $b_3$ each have only one positive entry, we further note that those entries have to appear in different positions, for otherwise we would have $a_{23}=0$.

Replacing $B$ with $BP$ and $C$ with $P^TCP$, where $P$ is a permutation matrix,  we may assume that $b_2^T=\alpha \npmatrix{1 & 0 & 0}$, $b_3^T=\beta \npmatrix{0 & 1 & 0}$, and $c_{11}=c_{22}=0$. Let $D$ be a diagonal matrix $D=\diag\npmatrix{\alpha & \beta & 1}$. Replacing $B$ by $BD^{-1}$ and $C$ by $DCD$, we may further assume that $\alpha=\beta=1$. From $a_{23}=1$, we now get $c_{12}=1$.

Since $B$ shares a column space with $B_1$ defined in \eqref{eq:ABC}, we have $B=B_1 X$, for $X \in \R^{3 \times 3}$. 
From the information that we already have on $B$, we deduce:
 $$X=\left(
\begin{array}{ccc}
 0 & 1 & 0 \\
 x_{21} & x_{22} & x_{23} \\
 1 & 0 & 0 \\
\end{array}
\right),$$
which in turn gives us:
$$B=\left(
\begin{array}{ccc}
 x_{21}+1 & x_{22} & x_{23} \\
 1 & 0 & 0 \\
 0 & 1 & 0 \\
 x_{21} & x_{22}+1 & x_{23} \\
\end{array}
\right).
$$
Again replacing $B$ with $BD^{-1}$ and $C$ with $DCD$, this time for matrix $D=\diag\npmatrix{1 & 1 & x_{23}}$, we get:
$$B=\left(
\begin{array}{ccc}
 x_{21}+1 & x_{22} & 1 \\
 1 & 0 & 0 \\
 0 & 1 & 0 \\
 x_{21} & x_{22}+1 & 1 \\
\end{array}
\right) \text{ and } 
C=\left(
\begin{array}{ccc}
 0 & 1 & c_{13} \\
 1 & 0 & c_{23} \\
 c_{13} & c_{23} & c_{33} \\
\end{array}
\right).$$
From $a_{21}=1$ and  $a_{31}=2$ we get $c_{13}=1-x_{22}$ and $c_{23}=1-x_{21}$, in particular showing $x_{21},x_{22} \in (0,1)$. Finally, $a_{11}=1$, gives us $c_{33}=-1 - 2 x_{21} - 2 x_{22} + 2 x_{21} x_{22} = 2(1-x_{21})(1-x_{22})-3$, which is negative for all $x_{21},x_{22} \in (0,1)$, a contradiction. Hence, $\spr(A) = 4 > \rk_+(A)=3$.
\end{example}

Note that in the example above we were not able to exclude $\spr(A)=3$ based on the pattern of $A$ alone. This example also allows us to show that the property of rank: $\rk(A^n)=\rk(A)$ when $A=A^T$, does not extend to $\spr$. 

\begin{example}
For the matrix $A$ from Example \ref{ex:2} we have
$$ A^2=\left(
\begin{array}{cccc}
 10 & 7 & 5 & 8  \\
 7 & 6 & 4 & 5  \\
 5 & 4 & 6 & 7  \\
 8 & 5 & 7 & 10 \\
\end{array}
\right) = BCB^T, $$
where
$$B=\left(
\begin{array}{ccc}
 0 & 1 & 1  \\
 0 & 0 & 1  \\
 1 & 0 & 0 \\
 1 & 1 & 0 \\
\end{array}
\right) \text{ and } 
C=\left(
\begin{array}{ccc}
 6 & 1 & 4 \\
 1 & 2 & 1 \\
 4 & 1 & 6 \\
\end{array}
\right),$$
so $\spr(A^2)=\rk(A^2)=3<\spr(A)=4$. 
\end{example}

For matrices with $\rk(A)=3$, $\spr(A)$ cannot be bounded by a constant independent of the size of the matrix $A$. This fact can be deduced from the equivalent statement for $\rk_+(A)$. This was first observed in \cite{MR2563025}, where it was shown that for the Euclidean distance matrix $M_n\in \mc S_n^+$, defined by $(M_n)_{ij} = (i-j)^2$, we have $\rk(M_n)=3$ but $\rk_+(M_n)$ cannot be bounded independently of $n$. The paper \cite{MR2964717} gives some lower bounds for NMF-rank of $M_n$ in Corollary 6, and the  upper bound $\rk_+(M_n) \le \lceil \frac{n}{2} \rceil +2$ in Theorem 9. This upper bound is proved by constructing a corresponding NM-Factorization, that we modify to SN-Trifactorization below. The NMF-rank for Euclidean distance matrices has been also considered in \cite{MR2905148,MR2653832, MR3918551}.

\begin{proposition}\label{edm}
Let $M_n\in \mc S_n^+$ be the matrix with $(M_n)_{ij} = (i-j)^2$.
Then $\spr(M_n) \le \lceil \frac{n}{2} \rceil+2$.
\end{proposition}

\begin{proof}
First suppose that $n$ is even. Let 
$$v = \left(\begin{array}{cccc}  1 & 3 & \ldots & n-1 \end{array} \right)^T \in \R^{\frac{n}{2}}$$ and
$K_n \in \R^{n \times n}$ be the matrix with ones on the anti-diagonal and zeros elsewhere.
For
$$B=\left(
\begin{array}{ccc}
 K_{\frac{n}{2}}v & 0 & I_{\frac{n}{2}} \\
 0 & v & K_{\frac{n}{2}} \\
\end{array}
\right), \ \ C=\left(
\begin{array}{ccc}
 0 & 1 & 0 \\
 1 & 0 & 0 \\
 0 & 0 & M_{\frac{n}{2}} \\
\end{array}
\right)$$ we get $M_n=BCB^T$. 
If $n$ is odd, $M_n$ is a principal submatrix of $M_{n+1}$, so $\spr(M_n) \le \spr(M_{n+1}) \le \lceil \frac{n+1}{2} \rceil+2 = 
\lceil \frac{n}{2} \rceil+2$.
\end{proof}

\section{Matrices whose SNT-rank equals rank} \label{sec:perturbation}

As we have seen, $\spr(A)$ can be significantly bigger than $\rk(A)$. In this section we take a closer look at the family of nonnegative symmetric matrices that satisfy $\rk(A)=\spr(A)$:
 $$\mc E_n:=\{A\in \mc S_n^+; \rk(A)=\spr(A)\}.$$

An invertible matrix $T$ is called a Perron similarity if one of its columns and the corresponding row of its inverse
are both nonnegative or both nonpositive. Perron similarities play a role in the theory of nonnegative matrices, as any matrix that brings an irreducible nonnegative matrix to its Jordan canonical form under similarity is Perron similarity. In proposition below we meet Perron similarities in connection with congruence transformation that connects a diagonal matrix with a nonnegative matrix $C$ in SN-Trifactorization $BCB^T$.

\begin{proposition}\label{thm:UT}
Let $A \in \mc E_n$ with $\rk(A)=\spr(A)=r$, the Perron eigenvalue $\lambda_1$ and the normalised Perron eigenvector  $u$.  Let
$$A=U(\lambda_1 \oplus D_1)U^T=BCB^T,$$
where $D_1$ is a diagonal matrix with nonzero entries on the diagonal,  $U=\npmatrix{u & U_1} \in \R^{n \times r}$, $U^TU=I_r$, $B \in \R_+^{n \times r}$ and $C \in \R_+^{r \times r}$.

Then there exists an invertible matrix $T$ with the first column of $T$ and the first row of $T^{-1}$ both nonnegative, so that $B=UT^{-1}$ and $C=T(\lambda_1 \oplus D_1)T^T$. If $A$ is irreducible then the first column of $T$ and the first row of $T^{-1}$ are both positive. 
\end{proposition}

  \begin{proof}
Let $A=BCB^T$, where $B \in \R_+^{n \times r}$ and $C \in \R_+^{r \times r}$. Since both $B$ and $U$ have rank $r$, and the span of columns of $B$ is equal to the span of columns of $U$, there exists an invertible matrix $T$ satisfying $B=UT^{-1}$. Now, $T^{-1}C(T^{-1})^T=\lambda_1 \oplus D_1$, hence $C=T(\lambda_1 \oplus D_1)T^T$. 

From, 
 $$ u^TB=u^TUT^{-1}=e_1^T T^{-1},$$
 where $e_1=\npmatrix{1 & 0 & \ldots & 0}^T$, we see that the first row of $T^{-1}$ is necessarily nonnegative. 
 Now 
 $$e_1 T^{-1}C=e_1 T^{-1}T(\lambda_1 \oplus D_1)T^T=\lambda_1 e_1^T T^T$$
 implies, that the first column of $T$ is nonnegative. 
 
 In the irreducible case, we know that $u$ is positive, and $B$ and $C$ have no columns equal to zero. The conclusion follows. 
 \end{proof}

If $u$ is the Perron eigenvector of a symmetric nonnegative matrix $A$ (normalised so that $u^Tu=1$), then it is straightforward to see that $\rk(A+\alpha uu^T)=\rk(A)$, and if we require  $\alpha \geq 0$, then clearly $A+\alpha uu^T$ remains nonnegative.    
The following theorem shows that if $\alpha$ is chosen to be sufficiently large, then $\spr(A+\alpha uu^T)$ drops down to $\rk(A)$. This type of perturbation was considered in connection with the completely positive rank in \cite{MR3365273}.

\begin{theorem}\label{prop:4.1}
Let $A\in \mc S_n^+$ be an irreducible symmetric nonnegative matrix with the Perron eigenvector $u$, $u^Tu=1$. Then $\spr(A+\alpha uu^T) \leq \spr(A)$ for all $\alpha \geq 0$, and there exists $\alpha_0$, so that $A+\alpha uu^T \in \mc E_n$ for all $\alpha \geq \alpha_0$. 
\end{theorem}

\begin{proof}
Let $A=BCB^T$, and $Au=\lambda_1 u$ with $u^Tu=1$. Direct calculation gives us:
$$(B+\beta uu^TB)C(B+\beta uu^TB)^T=A+\lambda_1(2 \beta+\beta^2)uu^T=A+\alpha uu^T,$$ for an appropriate choice of $\beta \geq 0$. This proves $\spr(A+\alpha uu^T) \leq \spr(A)$. 

To prove that there exists $\alpha$ with $\spr(A+\alpha uu^T) =\rk(A)=:r$, we start with a spectral decomposition of $A+\alpha uu^T$:
$$A+\alpha uu^T=\npmatrix{u & U_1}\npmatrix{\lambda_1+\alpha & 0 \\ 0 & D_1} \npmatrix{u^T \\ U_1^T},$$
where $D_1$ is a $(r-1)\times (r-1)$ diagonal matrix containing nonzero, non-Perron eigenvalues of $A$, and $U_1 \in \R^{n \times (r-1)}$ a matrix whose columns are equal to the corresponding normalised eigenvectors of $A$. 
Let $\beta >0$,  $q_1 \in \R_+^r$ be a positive vector satisfying $q_1^Tq_1=1$, and 
$Q=\npmatrix{q_1 & Q_1} \in \R^{r\times r}$ an orthogonal matrix. 
Let 
\begin{align*}
B(\beta)&:=\npmatrix{v & U_1}(\beta \oplus I_{r-1})Q^T,\\
 C(\alpha, \beta)&:=Q(\frac{1}{\beta} \oplus I_{r-1}) \npmatrix{\lambda_1+\alpha & 0 \\ 0 & D_1}(\frac{1}{\beta} \oplus I_{r-1})Q^T.
 \end{align*}
 Then $A+\alpha uu^T=B(\beta)C(\alpha, \beta)B(\beta)^T$ for all $\alpha, \beta >0.$ It remains to show that we can choose $\alpha >0$ and $\beta >0$ so that $B(\beta)>0$ and $C(\alpha, \beta)>0$. Note that:
 $$B(\beta)=\beta uq_1^T+U_1Q_1^T,$$
 and since $uq_1^T>0$ we can choose $\beta >0$ so that $B(\beta)>0$. On the other hand, we have:
 $$C(\alpha, \beta)=\frac{(\lambda_1+\alpha)}{\beta^2}q_1 q_1^T+Q_1D_1Q_1^T.$$
Since $q_1q_1^T>0$, we can choose $\alpha$ so that $C(\alpha, \beta)>0$ for any fixed $\beta>0$. 
\end{proof}

From Theorem \ref{prop:4.1} if follows that in order to understand $\mc E_n$, it is enough to study $\partial\mc E_n$, defined as the set of matrices $A\in \mc E_n$ with the property that $A-\alpha uu^T \not\in \mc E_n$ for any $\alpha >0$, where $u$ is the Perron eigenvector of $A$. In particular, all irreducible matrices in $\mc E_n$ that contain at least one zero entry are necessarily in $\partial \mc E_n$. Hence, given $A \in \mc S_n^+$, we would like to determine minimal $\alpha$ with $\rk(A+\alpha uu^T)=\spr(A+\alpha uu^T)$, or equivalently, we are looking for $\alpha$ with $A+\alpha uu^T \in \partial \mc E_n$. 

 From the proof of Theorem \ref{prop:4.1} we can produce upper bounds for $\alpha$ using different orthogonal matrices $Q$. In fact, in the proof, an orthogonal matrix $Q$ can be replaced by any $r \times r$ invertible matrix $S$ with the first column of $S$ and the first row of $S^{-1}$ both positive. Indeed, if we define:
\begin{align}\label{eq:BS}
B(\beta,S)&:=U(\beta \oplus I_{r-1})S^{-1},\\ \label{eq:CS}
 C(\alpha, \beta,S)&:=S(\frac{1}{\beta} \oplus I_{r-1}) \npmatrix{\lambda_1+\alpha & 0 \\ 0 & D_1}(\frac{1}{\beta} \oplus I_{r-1})S^T,
 \end{align}
 then  $A+\alpha vv^T=B(\beta,S)C(\alpha, \beta,S)B(\beta,S)^T$. As in the proof above, we can find $\beta$ that makes $B(\beta, S)$ nonnegative, and given $\beta$ and $S$ we can find $\alpha$ so that $C(\alpha,\beta,S)$ is nonnegative. Theorem \ref{thm:UT} implies that optimisation over all such invertible matrices $S$ will produce the optimal $\alpha$.  We explore this idea in Example \ref{ex:4.1}.

\begin{example} \label{ex:4.1}
The matrix
\begin{equation}\label{eq:AExample}
A=\left(
\begin{array}{cccc}
 0 & 2 & 1 & 1 \\
 2 & 0 & 1 & 1 \\
 1 & 1 & 0 & 2 \\
 1 & 1 & 2 & 0 \\
\end{array}
\right). 
\end{equation}
has the normalized Perron eigenvector  $v=\frac12\left(\begin{array}{cccc}  1 & 1 & 1 & 1 \end{array} \right)^T$, $\rk(A)=3$, and $\rk_+(A)=\spr(A)=4$. To show that $\rk_+(A)=4$, we can consider the Boolean rank of its derangement matrix
$$D=\left(
\begin{array}{cccc}
 0 & 1 & 1 & 1 \\
 1 & 0 & 1 & 1 \\
 1 & 1 & 0 & 1 \\
 1 & 1 & 1 & 0 \\
\end{array}
\right),$$
see \cite{MR2563025}. It is known that the Boolean rank of $D$ is $\min\{k, 4 \le \binom{k}{k/2}\}=4$, see \cite{MR657202}.

Below we consider three different invertible matrices $S_i$, $i=1,2,3$. For each of them, we first compute the minimal $\beta$ that makes $B(\beta,S)\geq 0$, using this optimal $\beta$ we then compute the minimal $\alpha$ that satisfies $C(\alpha,\beta,S) \geq 0$.

Taking $$S_1=\frac{1}{\sqrt{6}}\left(\begin{array}{ccc}  
\sqrt{2} & \sqrt{3} & 1 \\
\sqrt{2} & -\sqrt{3} & 1 \\
\sqrt{2} & 0 & -2
\end{array} \right)$$
gives us optimal  $\beta_1 = 2$ and $\alpha_1 = 12$. 
Hence, $\spr(A+12vv^T)=3$, as it is illustrated by $A+12vv^T =B(2,S_1)C(12,2,S_1)B(2,S_1)^T$ with
$$ B(2,S_1)=\frac{1}{2\sqrt{3}}\left(
\begin{array}{ccc}
 4 & 1 & 1  \\
 0 & 3 & 3 \\
 2 & 2+\sqrt{3} & 2-\sqrt{3} \\
 2 & 2-\sqrt{3} & 2+\sqrt{3} \\
\end{array}
\right) \text{ and } C(12,2,S_1)=\left(
\begin{array}{ccc}
 0 & 2 & 2 \\
 2 & 0 & 2 \\
 2 & 2 & 0 \\
\end{array}
\right). $$
Notice that this choice of $S_1$ yields an integer matrix $A+12vv^T$. 

Taking $$S_2=\frac{1}{2\sqrt{3}}\left(\begin{array}{ccc}  
2 & 2 & 2 \\
2 & -1+\sqrt{3} & -1-\sqrt{3} \\
2 & -1-\sqrt{3} & -1+\sqrt{3}
\end{array} \right)$$
gives us optimal  $\beta_2 = \frac12(\sqrt{6}-\sqrt{2})$ and $\alpha_2 = 4(1+\sqrt{3})\approx 10.92$. 
Thus  $\spr(A+4(1+\sqrt{3})vv^T)=3$, as can be illustrated by the SN-Trifactorization using the following factors:
$$ B({\textstyle \frac12(\sqrt{6}-\sqrt{2})},S_2)=\frac{1}{6\sqrt{2}}\left(
\begin{array}{ccc}
 3-\sqrt{3} & 6+2\sqrt{3} & 2\sqrt{3}  \\
 3+3\sqrt{3} & 0 & 6 \\
 3-\sqrt{3} & 2\sqrt{3} & 6+2\sqrt{3} \\
 3+3\sqrt{3} & 6 & 0 \\
\end{array}
\right)$$  and 
$$C(4(1+\sqrt{3}),{\textstyle \frac12(\sqrt{6}-\sqrt{2})},S_2)=\left(
\begin{array}{ccc}
 0 & 2 & 2 \\
 2 & 0 & 2 \\
 2 & 2 & 0 \\
\end{array}
\right). $$
This time the matrix $A+\alpha_2 vv^T$ is no longer an integer matrix, but $C(4(1+\sqrt{3}),{\textstyle \frac12(\sqrt{6}-\sqrt{2})},S_2)$ is. 

Finally, taking non-orthogonal
$$S_3^{-1}=\left(
\begin{array}{ccc}
 1 & 1 & 1 \\
 -1 & 1 & s \\
 -1 & s & 1 \\
\end{array}
\right),$$
where $s$ is the real root of  $p(s)=s^3+ s^2+5s+1$, approximately equal to $-0.207$, we get 
$$ B(\sqrt{2},S_3)=\frac{1}{\sqrt{2}}\left(
\begin{array}{ccc}
 0 & 1+s & 2  \\
 2 & 1-s & 0 \\
 0 & 2 & 1+s \\
 2 & 0 & 1-s \\
\end{array}
\right)$$
and  
$$C(\alpha_3,\sqrt{2},S_3)=\left(
\begin{array}{ccc}
 0 & \frac{2}{s^2+4s+3} & \frac{2}{s^2+4s+3} \\
 \frac{2}{s^2+4s+3} & 0 & \frac{4(2s^2+s+1)}{(s+3)(s^2-1)^2} \\
 \frac{2}{s^2+4s+3} & \frac{4(s^2+s+1)}{(s+3)(s^2-1)^2} & 0 \\
\end{array}
\right) \ge 0, $$
where $\alpha_3= \frac{4(-s^2-2s+1)}{(s^2-1)^2}$ is the real root of  $q(x)=x^3+ 2x^2-64x-256$, approximately equal to $8.71$. For this choice of $S_3$, and resulting $\alpha_3$ we loose integer entries in $A+\alpha_3 vv^T$ as well as in both $B(\sqrt{2},S_3)$ and $C(\alpha_3,\sqrt{2},S_3)$.
\end{example}

Ad hoc approach to find upper bounds for $\alpha$ demonstrated in Example \ref{ex:4.1} does not address the question, how to determine if $\alpha$ obtained is indeed optimal.

\begin{lemma}\label{lem:both positive}
Let $A\in \mc S_n^+$ with $\rk(A)=\spr(A)$ and $A=BCB^T$, where $B \in \R_+^{n \times r}$, $C \in \R_+^{r \times r}$, and at least one of the matrices $B$ or $C$ is positive. Then $A \not\in \partial \mc E_n$. 
\end{lemma}

\begin{proof}
Let $A=BCB^T$. Following the proof of Theorem \ref{prop:4.1} it is clear that $A \not\in \partial \mc E_n$, if both $B$ and $C$ are positive. Hence, we may assume that $B>0$ or $C>0$, but not both.  

 Let $Y \in \R_+^{r \times r}$ be positive and $T(\epsilon):=I-\epsilon Y$. For all sufficiently small $\epsilon >0$, $T(\epsilon)$ is an invertible M-matrix and hence $T(\epsilon)^{-1}>0$.  
 
If $B>0$, then $\hat B(\epsilon):=BT(\epsilon)>0$ for a sufficiently small $\epsilon >0$, and $\hat C(\epsilon):=T(\epsilon)^{-1}C(T(\epsilon)^{-1})^T>0$ for all $\epsilon >0$. Since $A=\hat B(\epsilon)\hat C(\epsilon)\hat B(\epsilon)^T$ we conclude $A \not\in \partial \mc E_n$. 

Similarly, if $C>0$, then $\hat C(\epsilon):=T(\epsilon)CT(\epsilon)^T>0$ for all sufficiently small $\epsilon >0$, and $\hat B(\epsilon):=BT(\epsilon)^{-1}$ for all $\epsilon >0$. The conclusion follows. 
\end{proof}

\begin{theorem}\label{thm:duality}
Let $A \in \mc E_n$ with $\rk(A)=r$, $A=BCB^T$, $B \in \R_+^{n \times r}$ and $C \in \R_+^{r \times r}$. If  the system: \begin{equation}
\begin{split}
 &X\in \R_+^{n \times r} \text{ and }W \in \R_+^{r \times r} \\
    & W=W^T \\ 
    & X \circ B=0, \, W \circ C=0, \\ 
    & WC=B^TX, 
\end{split}
\end{equation} implies $X=W=0$, then
 $A \not\in \partial \mc E_n$. 
\end{theorem}

\begin{proof}
By Lemma \ref{lem:both positive}, we know that, if the system of inequalities $BT > 0$ and $T^{-1}C(T^T)^{-1} >0$ has a solution for some invertible matrix $T$, then $A \not\in \partial\mc E_n$. In particular, if  $B(I-\epsilon Y) > 0$ and $(I-\epsilon Y)^{-1}C(I-\epsilon Y^T)^{-1} >0$ for some $Y \in \R^{r \times r}$ and $\epsilon >0$, then $A \not\in \partial \mc E_n$. 

Assume $A \in \partial \mc E_n$ and define $\mc Z(B):=\{(i,j); B_{ij}=0\}$ and $\mc Z(C):=\{(i,j); C_{ij}=0\}$. (Note that the assumption $A \in \partial \mc E_n$ implies $\mc Z(B)$ and $\mc Z(C)$ are not empty.)  From the formal expansion $(I-\epsilon Y)^{-1}=\sum_{i=0}^{\infty}(\epsilon Y)^i$, and looking at linear terms in $\epsilon$, we deduce that the system of linear inequalities:
\begin{align*}
    -(BY)_{ij}>0 \text{ for }(i,j) \in \mc Z(B)\\
    (YC+CY^T)_{ij}>0 \text{ for }(i,j) \in \mc Z(C)
\end{align*}
is not solvable for any $r \times r$ matrix $Y$. This system of linear inequalities is equivalent to the following one:
\begin{align*}
    &(I_k \otimes -B)[\mc Z(B),\_]\vc(Y)>0 \\
    &(C \otimes I_k +(I_k \otimes C)P)[\mc Z(C),\_]\vc(Y)>0,
\end{align*}
where $P$ is the permutation matrix satisfying $P\vc(Y)=\vc(Y^T)$. 
By the Transposition theorem of Gordan \cite{MR874114} this system is unsolvable if and only if the following dual system is solvable:
\begin{align*}
    & x \in \R_+^{|\mc Z(B)|}, z \in \R_+^{|\mc Z(C)|}, \text{ not both zero} \\
    & x^T (I_k \otimes -B)[\mc Z(B),\_]+z^T(C \otimes I_k +(I_k \otimes C)P)[\mc Z(C),\_]=0.
\end{align*}
Let $x$ and $z$ be solutions to the above, and let $x_0 \in \R^{r n}$ be the vector obtained from $x$ by inserting $(x_0)_i=0$ for $i$ that correspond to $\supp(B)$, i.e.  $x_0=\vc(X)$ with $X \circ B=0$. Similarly, let  $z_0 \in \R^{r^2}$ be the vector obtained from $z$ by inserting $(z_0)_i=0$ for $i$ that correspond to $\supp(C)$. In other words, $z_0=\vc(Z)$ satisfying $Z \circ C=0$.

The system above, rewritten in terms of $X$ and $Z$, becomes:
\begin{align*}
    & X\geq 0, Z\geq 0, \text{ not both zero} \\
    & X \circ B=0, \, Z \circ C=0 \\
    & ZC+Z^TC=B^TX.
\end{align*}
Introducing $W:=Z+Z^T$, we get: 
\begin{equation}\label{LPsystem}
\begin{split}
    & X\geq 0, W=W^T\geq 0, \text{ not both zero} \\ 
    & X \circ B=0, \, W \circ C=0 \\ 
    & WC=B^TX. 
\end{split}
\end{equation}
We have shown that the assumption $A \in \partial \mc E_n$ implies nonzero solution $(X,W)$ to the system \ref{LPsystem}. The conclusion of the theorem follows. 
\end{proof}

To illustrate how  Theorem \ref{thm:duality} is applied, we return to Example \ref{ex:4.1}.

\begin{example} \label{ex:4.1cont}
We assume all the notation and definitions from Example \ref{ex:4.1}. First we consider:
$$A+12vv^T =B(2,S_1)C(12,2,S_1)B(2,S_1)^T.$$
From Example \ref{ex:4.1}, we already know that $A+12vv^T \not\in \partial\mc E_n$. This is supported by the fact, that the only solution to the system \eqref{LPsystem} for $B(2,S_1)$ and $C(12,2,S_1)$ is $X=0$ and $W=0$. 

On the other hand, applying Theorem \ref{thm:duality} to 
$$A+\alpha_3 vv^T =B(\sqrt{2},S_3)C(\alpha_3,\sqrt{2},S_3)B(\sqrt{2},S_3)^T = B_3C_3B_3^T,$$
we get the following nonzero solutions to the system \eqref{LPsystem} for $B=B_3$ and $C=C_3$:  
$$ X = \left(
\begin{array}{ccc}
 x_{11} & 0 & 0 \\
 0 & 0 & x_{23} \\
 x_{31} & 0 & 0 \\
 0 & x_{42} & 0 
\end{array}\right), \ \  W=\left(
\begin{array}{ccc}
 w_{11} & 0 & 0 \\
 0 & w_{22} & 0 \\
 0 & 0 & w_{33}
\end{array}
\right)$$
with 
$$x_{11} = x_{31} = \frac{(C_3)_{12}}{(3+s)}w_{33}, \ \ 
x_{23} = x_{42} = \frac{(C_3)_{23}}{1-s}w_{33}, $$
$$w_{11} = \frac{2(C_3)_{23}}{(C_3)_{12}(1-s)}w_{33}, \ \ 
w_{22} = w_{33} .$$
While this does not prove $A+\alpha_3 vv^T \in \partial\mc E_n$, it does show that the factorisation $A=B_3C_3B_3^T$ cannot locally be moved to have positive factors.
\end{example}

\section{A Completion Problem} \label{sec:constructions}

For given nonnegative symmetric matrices $A_1$ and $A_2$, we consider the question of minimising the SNT-rank of 
\begin{equation}\label{eq:block}
    A=\npmatrix{A_1 & X \\ X^T & A_2},
\end{equation}
over all nonnegative matrices $X$ of appropriate order. 
We will consider two variants of this problem, one allowing any nonnegative $X$, and the other requiring $X$ to be positive.
Problems of this type occur in situations where only partial information on the data is known, and we desire unknown data to produce a matrix of low SNT-rank. Here our main motivation for considering this problem is to advance our understanding of matrices with low SNT-rank.

For $i=1,2$, let $A_i$ be $n_i \times n_i$ nonnegative symmetric matrices, and let $A$ be as above. We define:
\begin{align*}
    \spr(A_1,A_2)&:=\min\{\spr(A); A \text{ of the form }\eqref{eq:block} \text{ with } X \geq 0\}, \\
     \sprp(A_1,A_2)&:=\min\{\spr(A); A \text{ of the form }\eqref{eq:block} \text{ with } X > 0\}.
\end{align*}
The two ranks can happen to be the same for some given $A_1$ and $A_2$. The example below illustrates that  $\spr(A_1,A_2)<\spr^{>}(A_1,A_2)$ can also occur. 
 
 \begin{example}
Let $A_1=0$ and $A_2$ be a rank $1$ symmetric nonnegative matrix. Clearly, $\spr(A_1,A_2)=1$, and $\spr^>(A_1,A_2)=2$.
 \end{example}

Below we list some straightforward inequalities:
\begin{align*}
    \spr(A_1, A_2) &\leq \sprp(A_1,A_2), \\  \max\{\spr(A_1),\spr(A_2)\} &\leq \spr(A_1,A_2), \\ \spr(A_1,A_2)&\leq \spr(A_1)+\spr(A_2).
\end{align*}
 The last inequality holds, since we can always choose $X=0$.

 The corresponding question on low rank completion without nonnegativity constraints is resolved, and can be deduced from the main result in \cite{MR636217}. The solution depends on the inertia of matrices given on the block diagonal. 
 
\begin{definition}
Let $A \in \R^{n\times n}$ be a symmetric matrix. \emph{The inertia} of $A$ is the triple $\In(A)=(\pi_+,\pi_-,\pi_0)$, where $\pi_+$, $\pi_-$, $\pi_0$ are, respectively, the number of positive, negative and zero eigenvalues of $A$. 
\end{definition}

Inertia plays a role in the study of SN-Trifactorization $A=BCB^T$, since the interias of $A$ and $C$ are closely connected. Let $C \in \R^{k \times k}$ be a symmetric matrix, $B \in \R^{n \times k}$, $\In(C)=(\pi_+,\pi_-, \pi_0)$, and $\In(BCB^T)=(\pi_+',\pi_-', \pi_0')$.  Then it is well known,  \cite{MR2978290}, that $\pi_+ \geq \pi_+'$ and $\pi_- \geq \pi_-'$. Furthermore, if $B \in \R^{k \times k}$ is invertible, then $\In(C)=\In(BCB^T)$, .

Returning to the completion problem assume $A$ is as in \eqref{eq:block} with $\In(A)=(\pi,\nu,\delta)$ and $\In(A_i):=(\pi_i,\nu_i,\delta_i)$, $i=1,2$. Then  $\max\{\pi_1,\pi_2\}\leq \pi$ and $\max\{\nu_1,\nu_2\}\leq \nu$,  \cite{MR636217}. Hence, 
\begin{equation}\label{eq:intertia_bound}
\spr(A_1, A_2) \geq \max\{\pi_1,\pi_2\}+\max\{\nu_1,\nu_2\}.
\end{equation}

The following lemma, borrowed from the theory of the Schur complement, sheds some light into the connections between $A_1$ and $A_2$ that guarantee low $\spr(A_1,A_2)$.

\begin{lemma}\label{prop:schur}
Let $A_1 \in \mc S_n^+$, $A_0 \in \mc S_m^+$, $N \in \R_+^{n \times m}$, and 
$$A=\npmatrix{A_1 & A_1N \\ N^TA_1 & A_0+N^TA_1N}.$$
Then $\rk(A) = \rk(A_1)+\rk(A_0)$ and $\spr(A) \leq \spr(A_1)+\spr(A_0)$. \end{lemma}

\begin{proof} Let $A_1= B_1C_1B_1^T$ and $A_0= B_0C_0B_0^T$. Then
$$A=\npmatrix{B_1 & 0 \\ N^TB_1 & B_0 }\npmatrix{C_1 & 0 \\ 0 & C_0} \npmatrix{B_1^T & B_1^TN \\ 0 & B_0^T},$$
showing $\spr(A) \leq \spr(A_1)+\spr(A_0)$.
\end{proof}

\begin{corollary}\label{cor:Schur}
Let $A_1 \in \mc S_n^+$, an $n \times m$ nonnegative matrix, and
$$A=\npmatrix{A_1 & A_1N \\ N^TA_1 & N^TA_1N}.$$ 
Then $\spr(A)=\spr(A_1)$. 
Furthermore, $\spr(A_1,N^TA_1N)=\spr(A_1)$ for any $N \in \R_+^{n \times m}$.
\end{corollary}

Lemma \ref{prop:schur} gives us an approach to bound $\spr(A_1,A_2)$.  Indeed,
let $N$ be a nonnegative matrix with  $A_2-N^TA_1N$ nonnegative. Then: $\spr(A_1,A_2) \leq \spr(A_1)+\spr(A_2-N^TA_1N)$. 
The next example illustrates, that this inequality can be strict for all such $N$. 
 
 \begin{example}
Let
$$ A_1=\npmatrix{1 & 0 \\ 0 & 1} \text{ and }  A_2=\npmatrix{0 & 1 \\ 1 & 0}$$
with $\In(A_1)=(2,0,0)$ and $\In(A_2)=(1,1,0)$. 
From \eqref{eq:intertia_bound} we get $\spr(A_1,A_2) \geq 3$. This bound can be achieved by taking:
$$ X=\npmatrix{0 & 0 \\ a & \frac{1}{2a}},$$
where $a$ can be any positive number. Indeed:
$$ \npmatrix{A_1 & X \\ X^T & A_2}=\npmatrix{1 & 0 & 0 \\ 0 & 1 & \frac12 \\ 0 & 2a & 0 \\ 0 & 0 & \frac{1}{2a} } 
\npmatrix{1 & 0 & 0 \\ 0 & 0 & 1 \\ 0 & 1 & 0 }
\npmatrix{1 & 0 & 0 & 0 \\ 0 & 1 & 2a & 0 \\ 0 & \frac12 & 0 & \frac{1}{2a} } .$$
Observe, that $A_2-N^TA_1N=A_2-N^TN$ is nonnegative only for $N=0$. 

On the other hand, $\sprp(A_1,A_2) = 4$, since $\rk_+(A)=4$ for any choice of positive matrix $X$. 
Namely, suppose that 
$A=UV^T$ is a NM-Factorization of $A$ with $U,V \in \R_+^{4\times 3}$. 
Since each row of 
$A$ has a zero entry, each row of $U$ has to have one as well. Two rows of $U$ cannot have the same pattern of zeros, so at least one row of $U$ has two zeros. 
Without loss of generality we may assume that one of the rows of $U$, say $k$-th, equals $\npmatrix{1 & 0 & 0}$. It follows that the first row of $V$ equals 
the $k$-th row of $A$, so it contains three nonzero entries. It further follows that the first entry of each row of $A$, exept $k$-th, equals zero. 
So the rank of matrix $A$ with the $k$-th row omitted equals 2, a contradiction. \end{example}

The following lemma gives some insight into the case when $A$ is completed with a matrix $X$ of rank $1$.  
 
\begin{lemma}\label{prop:diag}
Let $\hat{A_1}=\npmatrix{A_1 & a \\ a^T & \alpha} \in \mc S_{n+1}^+$, and $A_2 \in \mc S_m^+$ with the Perron eigenvalue $\alpha$ and corresponding eigenvector $u$, $u^Tu=1$. Let
\begin{equation}\label{eq:diag}
 A=\npmatrix{A_1 & a u^T \\ ua^T & A_2}.
\end{equation}
Then:
\begin{enumerate}
\item $\rank (A)=\rank(\hat{A_1})+\rank(A_2)-1$ 
\item $\max\{\spr(\hat{A_1}),\spr(A_2)\} \leq \spr( A)\leq \spr(\hat{A_1})+\spr(A_2)$. 
\item If $\rank(A_2)=1$, then $\spr( A) = \spr(\hat{A_1})$.
\end{enumerate}
\end{lemma}

\begin{proof}
\begin{enumerate}
\item Proved in \cite[Lemma 5]{MR2098598}.
\item The upper bound is shown by constructing an SN-Trifactorization of $A$ from SN-Trifactorizations of $\hat{A_1}$ and $A_2$, as follows. Let
$$\hat{A_1}=\npmatrix{B_1 \\ b_1^T}C_1\npmatrix{B_1^T & b_1} \text{ and } A_2=B_2C_2B_2^T.$$ 
Then 
\begin{equation}\label{eq:SNT1}
A=\npmatrix{B_1 & 0 \\ 0 & B_2} \npmatrix{C_1 & \frac{1}{\alpha}C_1b_1 u^TB_2C_2  \\ \frac{1}{\alpha}C_2B_2^Tub_1^TC_1   & C_2} \npmatrix{B_1^T & 0 \\ 0 & B_2^T}.
\end{equation}
We have $\spr(A_2) \leq \spr(A)$ by Proposition \ref{prop:basic}. Finally, let 
$$BCB^T=\npmatrix{B_{11} \\ B_{21}}C \npmatrix{B_{11}^T & B_{21}^T}$$ be an SN-Trifactorization of $A$ that achieves $\spr(A)$, where the partition of $B$ respects the partition of $A$ in \eqref{eq:diag}. Then 
$$\hat A_1=\npmatrix{B_{11} \\ u^TB_{21}}C \npmatrix{B_{11}^T & B_{21}^Tu}$$
is the SN-Trifactorization of $\hat{A}$, proving $\spr(\hat A_1) \leq \spr(A)$.  
\item 
If $\rank(A_2)=1$, then $A_2=\alpha uu^T$. Let 
$$\hat{A_1}=\npmatrix{B_1 \\ b_1^T}C_1\npmatrix{B_1^T & b_1}.$$ 
Then
\begin{equation}\label{eq:SNT2}
A=\npmatrix{B_1 \\ ub_1^T} C_1 \npmatrix{B_1^T & b_1 u^T},
\end{equation}
hence $\spr(A) \leq \spr(\hat{A_1})$. The reverse inequality follows from 2.
\end{enumerate}
\end{proof}

\begin{corollary}
Let $A_i \in \mc S_{n_i}^+$, $i=1,2$, where $A_1$ has no zero rows, and $A_2$ has a positive eigenvector. Then $\spr^{>}(A_1,A_2) \leq \spr(A_1)+\spr(A_2).$ Moreover, if $\rank(A_2)=1$, then $\spr^{>}(A_1,A_2) =\spr(A_1,A_2)= \spr(A_1).$
\end{corollary}

\begin{proof}
Choose $b \in \R_+^n$ so that $A_1b$ is positive and $b^TA_1b$ is equal to the Perron eigenvalue of $A_2$. Let
$$\hat{A_1}=\npmatrix{A_1 & A_1b \\ b^TA_1 & b^TA_1b}. $$
By Corollary \ref{cor:Schur}, $\spr(\hat{A_1})=\spr(A_1)$, and using Lemma \ref{prop:diag}, we can construct $A$ with diagonal blocks $A_1$ and $A_2$, positive off-diagonal block, and $\spr(A) \leq \spr(A_1)+\spr(A_2)$. Hence, $\spr^>(A_1,A_2)\leq \spr(A_1)+\spr(A_2)$. 
If $\rank(A_2)=1$, then the same construction produces a matrix $A$ with $\spr^>(A_1,A_2)=\spr(A_1)$.
\end{proof}

Note that \eqref{eq:SNT1}, \eqref{eq:SNT2} give explicit SN-Trifactorizations corresponding to SNT-rank estimations in Proposition \ref{prop:diag}.

\begin{example}
The matrix $A$ in \eqref{eq:AExample} can be constructed by two applications of Lemma \ref{prop:diag} to $2 \times 2$ matrices as follows. Let $$\hat A_0=\npmatrix{2 & 2 \\ 2 & 2} \text{ and } A_2=\left(
\begin{array}{cc}
 0 & 2 \\
 2 & 0
\end{array}
\right).$$
The matrix $A_2$ has the Perron eigenvalue $2$ with corresponding eigenvector $u= \frac{1}{\sqrt{2}}\left(\begin{array}{cc}  1 & 1  \end{array} \right)^T$. Joining $\hat A_0$ and $A_2$ as in Proposition \ref{prop:diag} we get $$\hat A_1=\left(
\begin{array}{ccc}
 2 & \sqrt{2} & \sqrt{2} \\
 \sqrt{2} & 0 & 2 \\
 \sqrt{2} & 2  & 0
\end{array}
\right).$$
One more application of Proposition \ref{prop:diag}, this time joining $P\hat A_1P^T$ with $A_2$, where $P$ is a permutation matrix switching the first and the third row, gives us $A$. 

By \cite{MR3134266} this proves that $A$ belongs to a family of a nonnegative matrices that are generated by a Soules matrix. If a matrix generated by a Solues matrix happens to be positive semi-definite, then it is completely positive, and its cp-rank is equal to its rank, \cite{MR2139455}. 
If a matrix generated by a Soules matrix is not positive semi-definite, then is clearly not completely positive. The matrix $A$ in this example satisfies $\rk(A)=3 < \spr(A) =4$, hence we note that a symmetric matrix generated by Solues matrix can have SNT-rank bigger than rank. 

Furthermore, we have $\rk(A_1)=\spr(A_1)=2$, $\spr(A_2)=2$ and $\spr(A)=4$ from Example \ref{ex:4.1}. This shows that the upper bound in item 2. of Lemma \ref{prop:diag} cannot be improved, in general. 
\end{example}

Our last example generalizes Example \ref{ex:2}.

\begin{example}\label{ex:3}
Let us write the matrix $A$ from Example \ref{ex:2} as $$A=\hat{A}_1 = \left(
\begin{array}{cc}
 A_1 & a  \\
 a^T & 1  
\end{array}
\right), \text{ where } A_1 = \left(\begin{array}{ccc}
 1 & 1 & 2   \\
 1 & 0 & 1   \\
 2 & 1 & 0   
\end{array}
\right) \text{ and } a = \left(\begin{array}{c}
  2   \\  2   \\  1   
\end{array}
\right).$$
Let  $v \in \R^k_+$, $v^Tv=1$, and $A_2:=vv^T$. 
Inserting  $\hat{A}_1$ and $A_2$ in Lemma \ref{prop:diag} we can construct a matrix:
$$A(v)=\left(
\begin{array}{cc}
 A_1 & av^T  \\
 a^Tv & A_2  
\end{array}
\right) \in \R^{(k+3)\times(k+3)}_+$$ with
$\rk(A(v))=\rk_+(A(v)R)=3$ and $ \spr(A(v))=4$, giving us a family of matrices satisfying $\rk_+(A)<\spr(A)$. Indeed, $\spr(A(v))=\spr(\hat{A}_1)=4$ by item 3. of Lemma \ref{prop:diag}. Further,  $\rk_+(A(v))=3$ can  be deduced from the decomposition below:
$$A(v)=\left(
\begin{array}{ccl}
 0 & 1 & 1  \\
 0 & 0 & 1  \\
 1 & 0 & 0 \\
 v & v & 0_{k \times 1} \\
\end{array}
\right)\left(\begin{array}{cccl}
 2 & 1 & 0 & v^T  \\
 0 & 1 & 1 & 0_{1 \times k}  \\
 1 & 0 & 1 & 2v^T  \\
\end{array}
\right). $$
\end{example}

\section{Conclusion and open questions}

In this work we introduced the problem of SN-Trifactorization and SNT-rank, and developed some foundation results.  Since the SN-Trifactorization can be connected to both the NMF-factorization and the CP-Factori\-zation, research directions for further work can be easily found in the extensive literature on those factorizations. Here, we suggest a handful of questions that can be motivated by the results in this work.

\begin{enumerate}
\item In \cite{MR2964717}, the restricted nonnegative rank of a matrix $A$, denoted by $\rk_+^*(A)$, is defined to be the minimum value of $k$ such that there exist $U \in \R_+^{m \times k}$ and $V \in \R_+^{n \times k} $ satisfying $A=UV^T$ and $\rk(A)=\rk(U)$. Further, it is shown that $\rk_+(A)$ can be smaller than $\rk_+^*(A)$. Similarly, one can define $\spr^*(A)$ to be the minimal $k$ so that $A=BCB^T$ and $\rk(A)=\rk(B)$ for $B \in \R_+^{m \times k}$, $C=C^T \in \R^{k \times k}$. It would be interesting to explore to what extent the geometric interpretation of $\rk_+^*$ can be addapted to $\spr^*$, and find examples of matrices $A$ with $\spr(A)< \spr^*(A)$.  
\item Shitov \cite{MR3127681} found the bound $\rk_+(A) \leq \lceil \frac{6\min\{m,n\}}{7} \rceil$ for $A \in \R_+^{m \times n}$ with $\rk(A)=3$. On the other hand, Hannah and Laffey \cite{MR719859} and Barioli and Berman \cite{MR1969056} bounded the $\cp(A)$ in terms of $\rk(A)$, for a completely positive matrix $A$. In particular, they showed: $\cp(A)\leq \frac{\rk(A)(\rk(A)+1)}{2}-1$.
Proposition \ref{prop:basic1} implies that $\spr(A)$ has the same upper bound if $A$ is completely positive. From our discussion above it is clear that bounding $\spr(A)$ solely in terms of $\rk(A)$ for general symmetric nonnegative matrices is not possible. However, it would be interesting to explore if bounds similar to the one derived in \cite{MR3127681} can be found for $\spr(A)$. 
\item Starting with an integer (rational) matrix $A \in \mc S_n^+$, we may ask for SN-Trifactorization $A=BCB^T$, where $B$ and/or $C$ have integer (rational) entries.  This issue is touched upon in Example \ref{ex:4.1}, but is not thoroughly explored in this work. 
\end{enumerate}

\section*{Acknowledgments}
Damjana Kokol Bukov\v{s}ek acknowledges financial support from the Slovenian Research Agency (research core funding No. P1-0222).

\bibliographystyle{amsplain}
\bibliography{biblio}

\providecommand{\bysame}{\leavevmode\hbox to3em{\hrulefill}\thinspace}
\providecommand{\MR}{\relax\ifhmode\unskip\space\fi MR }
\providecommand{\MRhref}[2]{%
  \href{http://www.ams.org/mathscinet-getitem?mr=#1}{#2}
}
\providecommand{\href}[2]{#2}
\begin{thebibliography}{10}

\bibitem{pmlr-v28-arora13}
Sanjeev Arora, Rong Ge, Yonatan Halpern, David Mimno, Ankur Moitra, David
  Sontag, Yichen Wu, and Michael Zhu, \emph{A practical algorithm for topic
  modeling with provable guarantees}, Proceedings of the 30th International
  Conference on Machine Learning, vol.~28, 2013, pp.~280--288.

\bibitem{MR1969056}
F.~Barioli and A.~Berman, \emph{The maximal cp-rank of rank {$k$} completely
  positive matrices}, Linear Algebra Appl. \textbf{363} (2003), 17--33, Special
  issue on nonnegative matrices, $M$-matrices and their generalizations
  (Oberwolfach, 2000). \MR{1969056}

\bibitem{MR2563025}
LeRoy~B. Beasley and Thomas~J. Laffey, \emph{Real rank versus nonnegative
  rank}, Linear Algebra Appl. \textbf{431} (2009), no.~12, 2330--2335.
  \MR{2563025}

\bibitem{MR3414584}
Abraham Berman, Mirjam D\"{u}r, and Naomi Shaked-Monderer, \emph{Open problems
  in the theory of completely positive and copositive matrices}, Electron. J.
  Linear Algebra \textbf{29} (2015), 46--58. \MR{3414584}

\bibitem{MR1986666}
Abraham Berman and Naomi Shaked-Monderer, \emph{Completely positive matrices},
  World Scientific Publishing Co., Inc., River Edge, NJ, 2003. \MR{1986666}

\bibitem{MR3365273}
Immanuel~M. Bomze, Peter J.~C. Dickinson, and Georg Still, \emph{The structure
  of completely positive matrices according to their {CP}-rank and
  {CP}-plus-rank}, Linear Algebra Appl. \textbf{482} (2015), 191--206.
  \MR{3365273}

\bibitem{MR636217}
Bryan~E. Cain and E.~Marques~de S\'{a}, \emph{The inertia of a {H}ermitian
  matrix having prescribed complementary principal submatrices}, Linear Algebra
  Appl. \textbf{37} (1981), 161--171. \MR{636217}

\bibitem{MR1230356}
Joel~E. Cohen and Uriel~G. Rothblum, \emph{Nonnegative ranks, decompositions,
  and factorizations of nonnegative matrices}, Linear Algebra Appl.
  \textbf{190} (1993), 149--168. \MR{1230356}

\bibitem{MR657202}
D.~de~Caen, D.~A. Gregory, and N.~J. Pullman, \emph{The {B}oolean rank of
  zero-one matrices}, Proceedings of the {T}hird {C}aribbean {C}onference on
  {C}ombinatorics and {C}omputing ({B}ridgetown, 1981), Univ. West Indies, Cave
  Hill Campus, Barbados, 1981, pp.~169--173. \MR{657202}

\bibitem{ding-he-simon}
Chris Ding, Xiaofeng He, and Horst~D. Simon, \emph{On the equivalence of
  nonnegative matrix factorization and spectral clustering}, Proceedings of the
  2005 SIAM International Conference on Data Mining, 2005, pp.~606--610.

\bibitem{NIPS2003_1843e35d}
David Donoho and Victoria Stodden, \emph{When does non-negative matrix
  factorization give a correct decomposition into parts?}, Advances in Neural
  Information Processing Systems (S.~Thrun, L.~Saul, and B.~Sch\"{o}lkopf,
  eds.), vol.~16, MIT Press, 2004.

\bibitem{MR3134266}
Richard Ellard and Helena \v{S}migoc, \emph{Constructing new realisable lists
  from old in the {NIEP}}, Linear Algebra Appl. \textbf{440} (2014), 218--232.
  \MR{3134266}

\bibitem{8653529}
Xiao Fu, Kejun Huang, Nicholas~D. Sidiropoulos, and Wing-Kin Ma,
  \emph{Nonnegative matrix factorization for signal and data analytics:
  Identifiability, algorithms, and applications}, IEEE Signal Processing
  Magazine \textbf{36} (2019), no.~2, 59--80.

\bibitem{8338424}
Xiao Fu, Kejun Huang, Nicholas~D. Sidiropoulos, Qingjiang Shi, and Mingyi Hong,
  \emph{Anchor-free correlated topic modeling}, IEEE Transactions on Pattern
  Analysis and Machine Intelligence \textbf{41} (2019), no.~5, 1056--1071.

\bibitem{NMFGillis}
Nicolas Gillis, \emph{Nonnegative matrix factorization}, SIAM, 2020.

\bibitem{MR2964717}
Nicolas Gillis and Fran\c{c}ois Glineur, \emph{On the geometric interpretation
  of the nonnegative rank}, Linear Algebra Appl. \textbf{437} (2012), no.~11,
  2685--2712. \MR{2964717}

\bibitem{MR719859}
John Hannah and Thomas~J. Laffey, \emph{Nonnegative factorization of completely
  positive matrices}, Linear Algebra Appl. \textbf{55} (1983), 1--9.
  \MR{719859}

\bibitem{thesis-ho}
Ngoc-Diep Ho, \emph{Nonnegative matrix factorization algorithms and
  applications}, 2008, PhD thesis, Université catholique de Louvain.

\bibitem{MR2978290}
Roger~A. Horn and Charles~R. Johnson, \emph{Matrix analysis}, second ed.,
  Cambridge University Press, Cambridge, 2013. \MR{2978290}

\bibitem{MR2905148}
Pavel Hrube\v{s}, \emph{On the nonnegative rank of distance matrices}, Inform.
  Process. Lett. \textbf{112} (2012), no.~11, 457--461. \MR{2905148}

\bibitem{LSNature}
D.~Daniel Lee and H.~Sebastian Seung, \emph{Learning the parts of objects by
  non-negative matrix factorization}, Nature \textbf{401} (1999), 788--791.

\bibitem{leplat2021exact}
Valentin Leplat, Yurii Nesterov, Nicolas Gillis, and François Glineur,
  \emph{Exact nonnegative matrix factorization via conic optimization}, 2021.

\bibitem{MR2653832}
Matthew~M. Lin and Moody~T. Chu, \emph{On the nonnegative rank of {E}uclidean
  distance matrices}, Linear Algebra Appl. \textbf{433} (2010), no.~3,
  681--689. \MR{2653832}

\bibitem{https://doi.org/10.1002/env.3170050203}
Pentti Paatero and Unto Tapper, \emph{Positive matrix factorization: A
  non-negative factor model with optimal utilization of error estimates of data
  values}, Environmetrics \textbf{5} (1994), no.~2, 111--126.

\bibitem{MR874114}
Alexander Schrijver, \emph{Theory of linear and integer programming},
  Wiley-Interscience Series in Discrete Mathematics, John Wiley \& Sons, Ltd.,
  Chichester, 1986, A Wiley-Interscience Publication. \MR{874114}

\bibitem{MR2139455}
Naomi Shaked-Monderer, \emph{A note on the {CP}-rank of matrices generated by
  {S}oules matrices}, Electron. J. Linear Algebra \textbf{12} (2004/05), 2--5.
  \MR{2139455}

\bibitem{MR3273762}
Naomi Shaked-Monderer, Abraham Berman, Immanuel~M. Bomze, Florian Jarre, and
  Werner Schachinger, \emph{New results on the cp-rank and related properties
  of co(mpletely )positive matrices}, Linear Multilinear Algebra \textbf{63}
  (2015), no.~2, 384--396. \MR{3273762}

\bibitem{MR3127681}
Yaroslav Shitov, \emph{An upper bound for nonnegative rank}, J. Combin. Theory
  Ser. A \textbf{122} (2014), 126--132. \MR{3127681}

\bibitem{MR3720357}
\bysame, \emph{The nonnegative rank of a matrix: hard problems, easy
  solutions}, SIAM Rev. \textbf{59} (2017), no.~4, 794--800. \MR{3720357}

\bibitem{MR3918551}
\bysame, \emph{Euclidean distance matrices and separations in communication
  complexity theory}, Discrete Comput. Geom. \textbf{61} (2019), no.~3,
  653--660. \MR{3918551}

\bibitem{MR2252778}
Bart Vanluyten, Jan~C. Willems, and Bart De~Moor, \emph{Recursive filtering
  using quasi-realizations}, Positive systems, Lect. Notes Control Inf. Sci.,
  vol. 341, Springer, Berlin, 2006, pp.~367--374. \MR{2252778}

\bibitem{MR2098598}
Helena \v{S}migoc, \emph{The inverse eigenvalue problem for nonnegative
  matrices}, Linear Algebra Appl. \textbf{393} (2004), 365--374. \MR{2098598}

\end{thebibliography}

\end{document}